\documentclass[12pt]{article}
\usepackage{a4}
\usepackage{amsthm}
\usepackage{amsfonts}
\usepackage{amssymb}
\usepackage{amsmath}
\usepackage{cite}
\usepackage{epsfig}
\usepackage{enumitem}
\usepackage[T1]{fontenc}

\newtheorem{theorem}{Theorem}
\newtheorem{corollary}[theorem]{Corollary}
\newtheorem{lemma}[theorem]{Lemma}
\newcommand{\NN}{{\mathbb N}}
\newcommand{\RR}{{\mathbb R}}
\newcommand{\ZZ}{\ensuremath{\mathbb{Z}}}
\newcommand{\PP}{{\cal P}}
\newcommand{\dd}{\;\mbox{d}}
\newcommand{\cart}{\square}
\DeclareMathOperator{\sumh}{hom}
\DeclareMathOperator{\dist}{dist}

\usepackage{xcolor}
\definecolor{ForestGreen}{rgb}{0,0.45,0}
\definecolor{DistinguishableBlue}{rgb}{0.1,0.1,0.7451}
\usepackage[
  colorlinks,
  citecolor=DistinguishableBlue,
  linkcolor=ForestGreen,
]{hyperref}

\begin{document}
\title{The step Sidorenko property and non-norming edge-transitive graphs\thanks{The work of the first three authors has received funding from the European Research Council (ERC) under the European Union's Horizon 2020 research and innovation programme (grant agreement No 648509). This publication reflects only its authors' view; the European Research Council Executive Agency is not responsible for any use that may be made of the information it contains.}}

\author{Daniel Kr\'al'\thanks{Faculty of Informatics, Masaryk University, Botanick\'a 68A, 602 00 Brno, Czech Republic, and Mathematics Institute, DIMAP and Department of Computer Science, University of Warwick, Coventry CV4 7AL, UK. E-mail: {\tt dkral@fi.muni.cz}. The first author was also supported by the Engineering and Physical Sciences Research Council Standard Grant number EP/M025365/1.}\and
        Ta\'isa L.~Martins\thanks{Mathematics Institute and DIMAP, University of Warwick, Coventry CV4 7AL, UK. E-mail: {\tt t.lopes-martins@warwick.ac.uk}. This author was also supported by the CNPq Science Without Borders grant number 200932/2014-4.}\and
	P\'eter P\'al Pach\thanks{Department of Computer Science and DIMAP, University of Warwick, Coventry CV4 7AL, UK and Department of Computer Science and Information Theory, Budapest
  University of Technology and Economics, 1117 Budapest, Magyar tud\'osok
  k\"or\'utja 2., Hungary. E-mail: {\tt ppp@cs.bme.hu}. This author was also  supported by the National Research, Development and Innovation Office NKFIH
  (Grant Nr.~K124171).}\and
	Marcin Wrochna\thanks{Institute of Informatics, University of Warsaw, Poland. E-mail: {\tt m.wrochna@mimuw.edu.pl}.
	This author was supported by the National Science Centre of Poland grant number 2016/21/N/ST6/00475 and by the Foundation for Polish Science (FNP) via the START stipend programme. 
	The visit of this author to the University Warwick was supported by the Leverhulme Trust 2014 Philip Leverhulme Prize of the first author.}}

\date{} 
\maketitle

\begin{abstract}
Sidorenko's Conjecture asserts that every bipartite graph $H$ has the Sidorenko property,
i.e., a quasirandom graph minimizes the density of $H$ among all graphs with the same edge density.
We study a stronger property, 
which requires that
a quasirandom multipartite graph minimizes the density of $H$ among all graphs with the same edge densities between its parts;
this property is called the step Sidorenko property.
We show that many bipartite graphs fail to have the step Sidorenko property and
use our results to show the existence of a bipartite edge-transitive graph that is not weakly norming;
this answers a question of Hatami~[Israel J. Math. 175 (2010), 125--150].
\end{abstract}

\section{Introduction}

Sidorenko's Conjecture is one of the most important open problems in extremal graph theory.
A graph $H$ has the {\em Sidorenko property} 
if a quasirandom graph minimizes the density of $H$ among all graphs with the same edge density.
The beautiful conjecture of Erd\H os and Simonovits~\cite{Sim84} and of Sidorenko~\cite{Sid93}
asserts that every bipartite graph has the Sidorenko property (it is easy to see that non-bipartite
graphs fail to have the property).
In this paper, we consider a more general property, the step Sidorenko property, and
explore the link between this property and weakly norming graphs 
to show the existence of a bipartite edge-transitive graph that is not weakly norming.
This answers a question of Hatami~\cite{Hat10} whether such graphs exist.

Sidorenko's Conjecture is one of the central problems in extremal combinatorics.
Sidorenko~\cite{Sid93} confirmed the conjecture for trees, cycles and
bipartite graphs with one of the sides having at most three vertices;
the case of paths is equivalent to the Blakley-Roy inequality for matrices, which was proven in~\cite{BlaR65}.
Additional graphs, such as cubes and bipartite graphs with a vertex complete to the other part,
were added to the list of graphs with the Sidorenko property
by Conlon, Fox and Sudakov~\cite{ConFS10}, by Hatami~\cite{Hat10}, and by Szegedy~\cite{Sze17}.
Recursively described classes of bipartite graphs that have the Sidorenko property
were obtained by Conlon, Kim, Lee and Lee~\cite{ConKLLXX},
by Kim, Lee and Lee~\cite{KimLL16}, by Li and Szegedy~\cite{LiSXX} and by Szegedy~\cite{Sze14}.
In particular, Szegedy~\cite{Sze14} has described a class of graphs called thick graphs,
which are amenable to showing the Sidorenko property using the entropy method argument that he developed.
More recently, Conlon and Lee~\cite{ConL18} showed that
bipartite graphs such that one of the parts has many vertices of maximum degree have the Sidorenko property.
Sidorenko's Conjecture is also known to hold in the local sense~\cite[Proposition 16.27]{Lov12},
i.e., a small modification of a quasirandom graph preserving its edge density
does not decrease the number of copies of any bipartite graph.
A stronger statement of this type, which comes with uniform quantitative bounds,
has recently been proven by Fox and Wei~\cite{FoxW17}.

Sidorenko's Conjecture is also related to other well-studied problems in graph theory.
We would like to particularly mention the connection to quasirandom graphs.
We say that a graph $H$ is {\em forcing}
if all minimizers of the density of $H$ among graphs with the same edge density are quasirandom graphs.
Note that if $H$ is forcing, then $H$ has the Sidorenko property.
The classical result of Thomason~\cite{Tho87}, also see~\cite{ChuGW89},
says that the cycle of length four is forcing.
This result was generalized by Chung, Graham and Wilson~\cite{ChuGW89},
who showed that every complete bipartite graph $K_{2,n}$ is forcing, and
by Skokan and Thoma~\cite{SkoT04}, who showed that all complete bipartite graphs are forcing.
A characterization of forcing graphs was stated as a question by Skokan and Thoma~\cite{SkoT04} and
conjectured by Conlon, Fox and Sudakov~\cite{ConFS10}:
a graph $H$ is forcing if and only if $H$ is bipartite and contains a cycle.

Another graph theoretic notion related to Sidorenko's Conjecture is that of common graphs.
A graph $H$ is {\em common} if a quasirandom graph minimizes the sum of densities of $H$ and the complement of $H$.
An old theorem of Goodman~\cite{Goo59} says that the complete graph $K_3$ is common.
The conjecture of Erd\H os that the complete graph $K_4$ is also common
was disproved by an ingenious construction of Thomason~\cite{Tho89};
counterexamples with a simpler structure were found by Franek and R\"odl in~\cite{FraR93}.
Jagger, \v S\v tov\'\i{}\v cek and Thomason~\cite{JagST96} showed that no graph containing $K_4$ is common.
On the other hand, it is known that
the graph obtained from $K_4$ by removing an edge~\cite{FraR92} is common and
so is the wheel $W_5$~\cite{HatHKNR12}.
The classification of common graphs remains a wide open problem.

Our results are motivated by the relation of Sidorenko's Conjecture to \emph{weakly norming} graphs,
which are of substantial interest in the theory of graph limits.
Due to its technical nature, we defer the definition to Section~\ref{sec-prelim}.
Intuitively, these are graphs $H$ such that the density of $H$ in other graphs defines a norm on the space of graphons (graph limits). Chapter 14.1 in Lov\'asz' book~\cite{Lov12} gives an introduction to this notion.
Every weakly norming graph has the Sidorenko property~\cite{Hat10}.
However, every weakly norming graph also has a stronger property~\cite[Proposition 14.13]{Lov12},
which we call the {\em step Sidorenko property}.
Informally speaking, a graph $H$ has the step Sidorenko property
if a multipartite quasirandom graph minimizes the density of $H$ among all multipartite graphs with the same density
inside and between its parts; we give a formal definition in Section~\ref{sec-prelim}.
It is not hard to find a graph that has the Sidorenko property but not the step Sidorenko property;
the cycle of length four with an added pendant edge is an example (see Section~\ref{sec-prelim}).

In this paper,
we present techniques for showing that a bipartite graph fails to have the step Sidorenko property.
Our techniques allow us to show that graphs as simple and symmetric as toroidal grids, i.e., Cartesian products of any number of cycles, do not have the step Sidorenko property.
The only exceptions are hypercubes (and single cycles of even length), which were shown to be weakly norming by Hatami~\cite{Hat10} (see also~\cite[Proposition 14.2]{Lov12} for a concise presentation).
The fact that most of the toroidal grids are not weakly norming is surprising
when contrasted with the result of Conlon and Lee~\cite{ConL17} that
the incidence graphs of regular polytopes are weakly norming.
Since toroidal grids $C_n\cart C_n$ are edge-transitive,
this answers in the negative a question of Hatami~\cite{Hat10}
whether all edge-transitive bipartite graphs are weakly norming.

\section{Preliminaries}
\label{sec-prelim}

In this section, we introduce the notation that is used throughout the paper.
In general, we follow standard graph theory notation.
All graphs considered in this paper are simple and without loops.
We sometimes consider graphs with vertices and edges assigned non-negative weights;
when this is the case, we refer to such a graph as a {\em weighted graph}.
The order of a graph $G$, i.e., its number of vertices, will be denoted by $|G|$ and
the size of a graph $G$, i.e., its number of edges, by~$\|G\|$.
If $v$ and $w$ are two vertices of $G$,
then $\dist(v,w)$ is the distance between $v$ and $w$,
i.e., the number of edges of the shortest path from $v$ to $w$.
The \emph{Cartesian product} of graphs $G_1,\ldots,G_k$, denoted $G_1\cart\cdots\cart G_k$,
is the graph with vertex set equal to the Cartesian product of the vertex sets of $G_1,\ldots,G_k$,
where two vertices $(u_1,\ldots,u_k)$ and $(v_1,\ldots,v_k)$ are adjacent
if there exists $1\le i_0\le k$ such that $u_{i_0}v_{i_0}$ is an edge of $G_{i_0}$ and $u_i=v_i$ for all $i\not=i_0$.

In the rest of this section,
we introduce notation related to graph homomorphisms and
present notions from the theory of graph limits that we need for our exposition.
We also formally define the Sidorenko property, the step Sidorenko property and weakly norming graphs.

\subsection{Graph homomorphisms}

A {\em homomorphism} from a graph $H$ to a graph $G$ is a mapping $f$ from $V(H)$ to $V(G)$ such that
if $vv'$ is an edge of $H$, then $f(v)f(v')$ is an edge of $G$.
If $f$ is a homomorphism from $H$ to $G$,
$|f^{-1}(X)|$ for $X\subseteq V(G)$ denotes the number of vertices of $H$ mapped to a vertex in $X$ and
$|f^{-1}(X)|$ for $X\subseteq E(G)$ denotes the number of edges mapped to an edge in $X$;
for simplicity, we write $|f^{-1}(x)|$ instead of $|f^{-1}(\{x\})|$.

We will need to consider homomorphisms extending a partial mapping between vertices of $H$ and $G$ and
we now introduce notation that will be handful in this setting.
We write $H(v_1,\ldots,v_k)$ for a graph $H$ with $k$ distinguished vertices $v_1,\ldots,v_k$.
If $H(v_1,\ldots,v_k)$ and $G(v'_1,\ldots,v'_k)$ are two graphs with $k$ distinguished vertices,
then a homomorphism from $H(v_1,\ldots,v_k)$ to $G(v'_1,\ldots,v'_k)$ is
a homomorphism from $H$ to $G$ that maps $v_i$ to $v'_i$ for $i=1,\ldots,k$.

We will also consider homomorphisms to graphs with vertex and edge weights.
As given earlier,
a {\em weighted graph} is a graph $G$ where each vertex and each edge of $G$ is assigned a non-negative weight;
the mapping $w$ from $V(G)\cup E(G)$ assigning the weights will be referred to as a {\em weight function} of $G$.
The {\em weight of a homomorphism} $f$ from $H$ to a weighted graph $G$, denoted $w(f)$, is defined as
\[\prod_{v\in V(H)}w(f(v))\prod_{vv'\in E(H)}w(f(v)f(v'))=\prod_{v\in V(G)}w(v)^{|f^{-1}(v)|}\prod_{e\in E(G)}w(e)^{|f^{-1}(e)|}\;\mbox{.}\]
We will often speak about the sum of the weights of homomorphisms
from a graph $H(v_1,\ldots,v_k)$ to a weighted graph $G(v'_1,\ldots,v'_k)$;
this sum will be denoted by $\sumh(H(v_1,\ldots,v_k),G(v'_1,\ldots,v'_k))$ and
we will understand it to be zero if no such homomorphism exists.

We also use the just introduced notation for graphs with distinguished vertices when talking about blow-ups of graphs.
A {\em $k$-blow-up} of a graph $H(v)$ is the graph obtained from $H$ by replacing the vertex $v$ with $k$ new vertices,
which we refer to as {\em clones} of $v$.
The vertices different from $v$ preserve their adjacencies,
the clones of $v$ form an independent set and
each of them is adjacent precisely to the neighbors of $v$.
Observe that
if $H$ is a weighted graph, then if the edges of the $k$-blow-up of $H(v)$ have the same weight as in $H$,
the vertices of the $k$-blow-up except for the clones have the same weights as in $H$ and
each clone has weight equal to $1/k$ of the weight of $v$,
then the sum of the weights of homomorphisms from $G$ to $H$ and
the sum of the weights of homomorphisms from $G$ to the $k$-blow-up are the same for every graph $G$.

\subsection{Graph limits}

The theory of graph limits offers analytic tools to study large graphs.
We present here only those notions that we need further, and
refer the reader to the monograph of Lov\'asz~\cite{Lov12} on the subject
for a comprehensive introduction to the theory.

Let $t(H,G)$ be the normalized number of homomorphisms from a graph $H$ to a graph $G$,
i.e., $t(H,G)=\sumh(H,G)/|V(G)|^{|V(H)|}$
where $G$ in $\sumh(H,G)$ is understood to have all the vertex and edge weights equal to one.
A sequence $(G_n)_{n\in\NN}$ of graphs is {\em convergent}
if the sequence $t(H,G_n)$ converges for every graph $H$.
A convergent sequence of graphs can be represented by an analytic object called a graphon.
A {\em graphon} is a (Borel) measurable symmetric function $W$ from $[0,1]^2$ to $[0,1]$,
i.e., $W(x,y)=W(y,x)$ for all $(x,y)\in [0,1]^2$.
One can think (although very imprecisely) of a graphon as a continuous version of the adjacency matrix of a graph.
Led by this intuition, we can define the density of a graph $H$ in a graphon $W$ as
\[ t(H,W)=\int_{[0,1]^{V(H)}}\prod_{vv'\in E(H)}W(x_{v},x_{v'}) \dd x^{V(H)}\;\mbox{.}\]
Note that the definition of $t(H,W)$ does not require $W$ to be non-negative and
we can define $t(H,f)$ in the same way for any bounded measurable function $f:[0,1]^2\to \RR$.

We say that a graphon $W$ is a {\em limit} of a convergent sequence $(G_n)_{n\in\NN}$ of graphs
if $t(H,W)$ is the limit density of $t(H,G_n)$ for every graph $H$.
It is not hard to show that for every graphon $W$, there exists a convergent sequence of graphs such that $W$ is its limit.
The converse statement is also true as shown by Lov\'asz and Szegedy~\cite{LovS06},
i.e., for every convergent sequence of graphs, there exists a graphon that is its limit.

The density $t(K_2,W)$ of $K_2$ is equal to the $L_1$-norm of a graphon $W$ as a function from $[0,1]^2$.
This leads to the question which graphs $H$ can be used to define a norm on the space of measurable functions on $[0,1]^2$ or,
more restrictively, on the space of graphons.
That is, we say that a graph $H$ is {\em weakly norming}
if the function $\|W\|_{H}=t(H,W)^{1/\|H\|}$ is a norm on the space of graphons,
i.e., $\|W\|_{H}=0$ if and only if $W$ is equal to zero almost-everywhere and the triangle inequality
$\|W_1+W_2\|_{H}\le \|W_1\|_{H}+\|W_2\|_{H}$ holds for any two graphons $W_1$ and $W_2$.
Observe that $H$ is weakly norming if and only if
$\|\ |f|\ \|_{H}$ is a norm on the set of all bounded symmetric functions $f$ from $[0,1]^2$ to $\RR$ (if we required that $\|f\|_{H}$, without the absolute value, is a norm on all such functions, we would get the slightly stronger notion of \emph{norming graphs}).

It is not hard to show that every weakly norming graph must be bipartite.
Hatami~\cite{Hat10} showed stronger statements as corollaries
of his characterization of weakly norming graphs as those satisfying a certain H\"older type inequality.
First, every weakly norming graph $H$ must be biregular,
i.e., all vertices in the same part of its bipartition have the same degree.
Second,
every subgraph $H'$ of a connected weakly norming graph $H$ must satisfy that
\[\frac{\|H'\|}{|H'|-1}\le\frac{\|H\|}{|H|-1}\;\mbox{.}\]
Weakly norming graphs include complete bipartite graphs (in particular, stars), even cycles and hypercubes~\cite{Hat10};
later, Conlon and Lee~\cite{ConL17} presented a large class of weakly norming graphs, which they refer to as reflection graphs.

Every weighted graph $G$ with a weight function $w$ that assigns edges weights between $0$ and $1$ can be associated with a graphon $W_G$ as follows.
Each vertex $v$ of $G$ is associated with a measurable set $J_v$ with measure $w(v)/w(V(G))$ in such a way that
the sets $J_v$, $v\in V(G)$, form a partition of the interval $[0,1]$;
$w(V(G))$ denotes the sum of the weights of the vertices of $G$.
For $x\in J_v$ and $y\in J_{v'}$,
we set $W(x,y)=w(vv')$ if $vv'\in E(G)$ and $W(x,y)=0$ otherwise (in particular, we set $W(x,y)=0$ if $v=v'$).
It is not hard to observe that $\sumh(H,G)$ is equal to $t(H,W_G)\cdot w(V(G))^{|H|}$;
in particular, if the sum of the weights of vertices of $G$ is one, then $\sumh(H,G)=t(H,W_G)$.
This correspondence will allow us to study weakly norming graphs in terms of weighted homomorphisms.

\subsection{Step Sidorenko property}

We now use the language of graph limits to describe the Sidorenko property and
to formally define the step Sidorenko property.
A graph $H$ has the {\em Sidorenko property} if
\begin{equation}
t(K_2,W)^{\|H\|}\le t(H,W) \label{eq-Sidprop}
\end{equation}
for every graphon $W$.
The left hand side can also be written as $t(H, U_p)$, where $U_p \equiv p$ is the constant graphon with the same edge density $p = t(K_2,W)$ as~$W$.
A graph $H$ is {\em forcing} if it has the Sidorenko property and
 \eqref{eq-Sidprop} holds with equality only if $W$ is equal to some $p\in [0,1]$ almost everywhere.
As we have presented earlier,
Sidorenko's Conjecture asserts that every bipartite graph has the Sidorenko property and
the Forcing Conjecture asserts that every bipartite graph with a cycle is forcing.

Let $\PP$ be a partition of the interval $[0,1]$ into finitely many non-null measurable sets.
We now define the {\em stepping operator\/}.
If $W$ is a graphon, then the graphon $W^{\PP}$ is defined for $(x,y)\in [0,1]^2$ as the `step-wise average':
\[W^{\PP}(x,y)=\frac{1}{|J\|J'|}\int_{J\times J'} W(s,t)\dd s\dd t\]
where $J$ and $J'$ are the unique parts from $\PP$ such that $x\in J$ and $y\in J'$, and
$|X|$ denotes the measure of a measurable subset $X\subseteq [0,1]$.
Note that the graphon $W^{\PP}$ is constant on $J\times J'$ for any $J,J'\in\PP$,
i.e., the graphon $W^{\PP}$ is a step graphon.

Let $\PP_0$ be the partition with a single part being the interval $[0,1]$ itself.
A graph $H$ has the Sidorenko property if and only if $t(H,W^{\PP_0})\le t(H,W)$ for every graphon $W$.
This motivates the following definition.
A graph $H$ has the {\em step Sidorenko property} if and only if
\[t(H,W^{\PP})\le t(H,W)\]
for every graphon $W$ and every partition $\PP$ of $[0,1]$ into finitely many non-null measurable sets.
Since all weakly norming graphs~\cite[Proposition 14.13]{Lov12} have the step Sidorenko property,
it follows that complete bipartite graphs, even cycles, hypercubes and more generally reflection graphs
defined by Conlon and Lee~\cite{ConL17} all have the step Sidorenko property.

The definition of the step Sidorenko property yields that
every graph that has the step Sidorenko property also has the Sidorenko property.
However, the converse is not true in general as we now demonstrate.
Let $C_4^+$ be the $5$-vertex graph obtained from a cycle of length four 
by adding a single vertex adjacent to one of the vertices of the cycle.
The graph $C_4^+$ has the Sidorenko property because, e.g.,
it is a bipartite graph with a vertex complete to the other part~\cite{ConFS10}.
On the other hand, $C_4^+$ does not have the step Sidorenko property.
Consider the partition $\PP=\{[0,\frac{2}{5}),[\frac{2}{5},1]\}$ and
the graphon $W$ that is defined as follows (the symmetric cases of $(x,y)$ are omitted).
\[ W(x,y) =\begin{cases}
 0.9 & \mbox{if $(x, y) \in [0,\frac{1}{5})\times [0,\frac{1}{5})$,}\\
 0.85& \mbox{if $(x, y) \in [0,\frac{1}{5})\times [\frac{1}{5},\frac{2}{5})$,}\\
 0.2& \mbox{if $(x, y) \in [0,\frac{1}{5})\times [\frac{2}{5},1]$,}\\
 1 & \mbox{if $(x, y) \in [\frac{1}{5},\frac{2}{5})\times [\frac{1}{5},\frac{2}{5})$, and}\\
 0 & \mbox{otherwise.}
\end{cases} \]
A straightforward computation\footnote{We thank Adam Finchett and Jonathan Noel for indicating an error in previous calculations.\looseness=-1} yields that
\begin{align*}
t(C_4^+,W)&\simeq 0.007453\mbox{ and}\\
t(C_4^+,W^{\PP})&\simeq 0.007461 > t(C_4^+,W)\;\mbox{.}
\end{align*}
Hence, the graph $C_4^+$ does not have the step Sidorenko property.

\section{Grids}

In this section, we demonstrate our techniques from Section~\ref{sec-general} in a less general setting.
We believe that this makes our presentation more accessible.

Intuitively, we consider a graph $G$ with distinguished vertices $u_0, u_1, u_2$ such that $u_0 u_1$ and $u_0 u_2$ are edges.
The idea is to blow-up $u_0$ into two copies and slightly perturb weights
only on edges corresponding to $u_0 u_1$ and $u_0 u_2$,
increasing weights of edges for one copy and
decreasing it for the other proportionally to a parameter $\alpha$,
resulting in a weighted graph $G_\alpha$.
A partition $\PP$ on the corresponding graphon $W_\alpha$ is then defined
so that the stepping operator averages out this perturbation,
returning to the original graph: $W_{\alpha}^{\PP}=W_G$.
The difference in homomorphism densities $t(H,W_\alpha^\PP)-t(H,W_\alpha)$
is then analyzed in the limit of small perturbations $\alpha$:
first order changes (those linear in $\alpha$) cancel out.
Second order changes result in a condition that can be expressed fairly concisely
as positive semidefiniteness of a matrix whose entries count certain constrained homomorphisms.

The more powerful setting in Section~\ref{sec-general} uses essentially the same idea, only blowing up more vertices,
resulting in a larger matrix and allowing us to further constraint the homomorphisms we have to count.
We turn to choosing the starting weighted graph $G$ and interpreting these counts in later corollaries.

\begin{theorem}
\label{thm-weights}
Let $H$ be a graph and
let $G$ be a weighted graph with three distinguished vertices $u_0$, $u_1$ and $u_2$ such that $u_0u_1$ and $u_0u_2$ are edges.
For $i,j\in\{1,2\}$, 
let $M_{ij}$ be the sum of the weights of homomorphisms from $H(v_0,v_1,v_2)$ to $G(u_0,u_i,u_j)$
summed over all choices of vertices $v_0$, $v_1$ and $v_2$ in $H$ such that $v_0v_1$ and $v_0v_2$ are edges, i.e.,
\[M_{ij}=\sum_{v_0v_1,v_0v_2\in E(H)}\sumh(H(v_0,v_1,v_2),G(u_0,u_i,u_j))\;\mbox{.}\]
If the $(2\times 2)$-matrix $M$ is not positive semidefinite, i.e., $M_{11} M_{22} < {M_{12}}^2$, then $H$ does not have the step Sidorenko property.
\end{theorem}

\begin{proof}
Let $w$ be the weight function of $G$.
We assume that the sum of the weights of vertices of $G$ is one (if needed, we multiply the weights of all vertices by the same constant).
Consider the step graphon $W_G$ associated with the weighted graph $G$.
Let $J_u$ be the measurable set corresponding to a vertex $u$ of $G$ and set $\PP=\{J_u,u\in V(G)\}$.

Suppose that the matrix $M$ associated with $G$ is not positive semidefinite and fix a vector $a=(a_1,a_2)^T$ such that $a^TMa<0$.
We next define a weighted graph $G_{\alpha}$ with a parameter $\alpha \geq 0$ as follows.
The graph $G_{\alpha}$ is a $2$-blow-up of $G(u_0)$; let $u^+_0$ and $u^-_0$ be the clones of $u_0$.
Each of the clones $u^+_0$ and $u^-_0$ has weight $w(u_0)/2$.
The weight of the edge $u^+_0u_i$ is $w(u_0u_i)(1+\alpha a_i)$ and
the weight of the edge $u^-_0u_i$ is $w(u_0u_i)(1-\alpha a_i)$, $i=1,2$.
The remaining vertices and edges have weights equal to their counterparts in $G$.
Let $W_\alpha$ be the step graphon associated with the weighted graph $G_{\alpha}$ such that
the set corresponding to a vertex $u\not=u_0$ is $J_u$ and
the sets corresponding to the vertices $u^+_0$ and $u^-_0$ are subsets of $J_{u_0}$.
Observe that $W_G=W_\alpha$ for $\alpha=0$ and that $W_G=W_\alpha^\PP$ for any $\alpha$.

Our aim is to show that $t(H,W_\alpha)<t(H,W_G)$ for some $\alpha\in(0,1)$.
To do so, we analyze the density $t(H,W_\alpha)$ as a function of $\alpha$.
Note that $t(H,W_\alpha)$ is actually a polynomial in $\alpha$.
We next wish to determine the coefficients $c_1$ and $c_2$ such that
\begin{equation}
t(H,W_\alpha)=t(H,W_G)+c_1\alpha+c_2\alpha^2+O(\alpha^3)\;\mbox{.}\label{eq-weights-Taylor}
\end{equation}
The coefficient $c_1$ can be determined as follows:
\begin{align*}
c_1 = \sum\limits_{v_0v_1\in E(H)} &
a_1\sumh(H(v_0,v_1),G_0(u^+_0,u_1))-
a_1\sumh(H(v_0,v_1),G_0(u^-_0,u_1))+{}\\ &
a_2\sumh(H(v_0,v_1),G_0(u^+_0,u_2))-
a_2\sumh(H(v_0,v_1),G_0(u^-_0,u_2))
\;\mbox{.}
\end{align*}
Since $\sumh(H(v_0,v_1),G_0(u^+_0,u_i))=\sumh(H(v_0,v_1),G_0(u^-_0,u_i))$ for all edges $v_0v_1\in E(G)$ and all $i\in\{1,2\}$,
we conclude that $c_1=0$.

We next analyze the coefficient $c_2$.
In this case,
we need to count homomorphisms mapping two edges, say $v_0v_1$ and $v'_0v'_1$, of $H$ to edges $u^+_0u_i$ and to $u^-_0u_i$ of $G_0$, $i=1,2$.
If $v_0\not=v'_0$, then the contributions of the homomorphisms mapping the edge $v_0v_1$ to $u^+_0u_i$ and $u^-_0u_i$ have opposite signs and cancel out.
Hence, we obtain the following formula for $c_2$:
\begin{align*}
c_2 = \sum\limits_{v_0v_1,v_0v_2\in E(H)}\;\sum\limits_{i,j=1}^2a_ia_j&\left(\sumh(H(v_0,v_1,v_2),G_0(u^+_0,u_i,u_j))\right.+\\
&\;\left.\sumh(H(v_0,v_1,v_2),G_0(u^-_0,u_i,u_j))\right)\;\mbox{.}
\end{align*}
The definition of the matrix $M$ now yields that
\[
c_2 = \sum_{i,j=1}^2a_ia_j\cdot M_{ij}=a^TMa<0\;\mbox{.}
\]
Since $c_1=0$ and $c_2<0$, we conclude using $W_G=W_\alpha^\PP$ and \eqref{eq-weights-Taylor} that $t(H,W_\alpha)<t(H,W_G)$ for small enough $\alpha>0$.
It follows that the graph $H$ does not have the step Sidorenko property.
\end{proof}

The setting of Theorem~\ref{thm-weights} is sufficient to prove that
the only two-dimensional toroidal grid that is weakly norming is $C_4\cart C_4$ (note that
the toroidal grids $C_{\ell}\cart C_{\ell}$ with $\ell$ odd
are not Sidorenko, and hence also not weakly norming, because they are not bipartite).

We apply Theorem~\ref{thm-weights} with $G=H=C_\ell\cart C_\ell$.
The identity homomorphism contributes to the off-diagonal entry of the matrix from Theorem~\ref{thm-weights}
while the homomorphisms contributing to the diagonal entries have to ``fold'' two edges onto one.
We choose weights in the target grid in such a wat that
the contribution of the former homomorphisms becomes smaller,
which makes the matrix not to be positive semidefinite.

\begin{corollary}
\label{cor-grid}
Let $\ell\ge 6$ be an even integer.
The Cartesian product $C_\ell\cart C_\ell$ does not have the step Sidorenko property.
\end{corollary}

\begin{figure}[t]
\begin{center}
\epsfbox{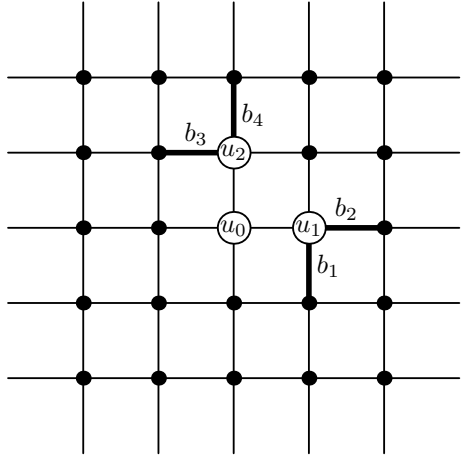} 
\end{center}
\caption{Notation used in the proof of Corollary~\ref{cor-grid}. The edges $b_1$, $b_2$, $b_3$ and $b_4$ are drawn bold.}
\label{fig-grid}
\end{figure}

\begin{proof}	
Fix $\ell\ge 6$ and let $G$ and $H$ be both equal to the graph $C_\ell \; \cart \; C_\ell$;
we denote the vertices of $G$ and $H$ by $(i,j)$, $0\le i,j\le \ell-1$, in such a way that
two vertices are adjacent if they agree in one of the coordinates and differ by one in the other (all computations
with the entries are computed modulo $\ell$ throughout the proof).
Let $u_0$ be the vertex $(0,0)$, $u_1$ the vertex $(1,0)$ and $u_2$ the vertex $(0,1)$.
Further, let $b_1$ be the edge $(1,0)(1,-1)$, $b_2$ the edge $(1,0)(2,0)$,
$b_3$ the edge $(0,1)(-1,1)$ and $b_4$ the edge $(0,1)(0,2)$ (see Figure~\ref{fig-grid}).

We next define the weights of the vertices and the edges of $G$;
to do so, we use a parameter $\gamma\in\NN$, which will be fixed later.
The weight $w(v)$ of a vertex $v$ is $\gamma^{\dist(u_0,v)}$ for $v\not=u_0,u_1,u_2$,
$w(u_0)= \gamma^{-3}$ and $w(u_i)=\gamma^{\dist(u_0,u_i)-3}=\gamma^{-2}$, $i=1,2$.
The weights of all edges of $G$ are equal to one except for the edges $b_1$, $b_2$, $b_3$ and $b_4$ that have weight $\gamma^{-1/4}$.

We wish to apply Theorem~\ref{thm-weights} with the graphs $H$ and $G$, and the distinguished vertices $u_0$, $u_1$ and $u_2$.	
Instead of verifying that the matrix $M$ from the statement of Theorem~\ref{thm-weights} is not positive semidefinite,
we consider the matrix $M$ such that
\[M_{ij}=\sum_{v_1,v_2\in N_H(u_0)}\sumh(H(u_0,v_1,v_2),G(u_0,u_i,u_j))\;\mbox{.}\]
Since $H$ is vertex-transitive,
the considered matrix $M$ is positive semidefinite if and only if
the matrix from the statement of Theorem~\ref{thm-weights} is.
Observe that $M_{1,1}=M_{2,2}$ and $M_{1,2}=M_{2,1}$.

Consider a homomorphism $f$ from $H(u_0,v_1,v_2)$ to $G(u_0,u_i,u_j)$ for some $i,j \in \{1,2\}$.
Observe that the weight of the homomorphism $f$ is equal to
\[\gamma^{\sum\limits_{v \in V(H)}\dist(u_0,f(v))-3\left|f^{-1}(\{u_0,u_1,u_2\})\right|-\frac{1}{4}\left|f^{-1}(\{b_1,b_2,b_3,b_4\})\right|}\;\mbox{.}\]
Note that if $f$ is the identity, then the weight of $f$ is equal to $\gamma^W$ where
\[W=\sum_{v \in V(H)} \dist(u_0,v)-10\;\mbox{.}\]
Since the identity is a homomorphism from $H(u_0,u_i,u_j)$ to $G(u_0,u_i,u_j)$ for $i\not=j$,
it follows that the entries $M_{1,2}$ and $M_{2,1}$ are of order $\Omega(\gamma^W)$, as functions of $\gamma$.
We next show that both $M_{1,1}$ and $M_{2,2}$ are of order $o(\gamma^W)$.
Since $M_{1,1}=M_{2,2}$, it is enough to argue that that $M_{1,1}=o(\gamma^W)$.

We show that every homomorphism $f$ from $H(u_0,v_1,v_2)$ to $G(u_0,u_1,u_1)$ has weight at most $\gamma^{W-\frac{1}{2}}$;
this will imply that $M_{1,1}=o(\gamma^W)$.
Fix a homomorphism $f$ from $H(u_0,v_1,v_2)$ to $G(u_0,u_1,u_1)$ with weight at least $\gamma^W$.
By symmetry, we may assume that $v_1=(1,0)$ and $v_2\in\{(-1,0),(0,1)\}$.
Note that $\left|f^{-1}(\{u_0,u_1,u_2\})\right|\ge 3$.
Since $f$ is a homomorphism,
any shortest path from $u_0$ to $v$ is mapped by $f$ to a walk of at most length $\dist(u_0,v)$ from $f(u_0)=u_0$ to $f(v)$,
it follows that $\dist(u_0,f(v))\leq\dist(u_0,v)$ for every vertex $v$.
Also observe that the parities of $\dist(u_0,f(v))$ and $\dist(u_0,v)$ are the same since the graph $G=H$ is bipartite.
Since the weight of $f$ is at least $\gamma^W$,
the following holds:
$\left|f^{-1}(\{u_0,u_1,u_2\})\right|=3$, $\dist(u_0,f(v))=\dist(u_0,v)$ for every vertex $v$ of $H$ and $\left|f^{-1}(\{b_1,b_2, b_3, b_4\})\right|\le 4$.
Since $\left|f^{-1}(\{u_0,u_1,u_2\})\right|=3$, 
no vertex other than $u_0$, $v_1$ and $v_2$ is mapped by $f$ to any of $u_0$, $u_1$ and $u_2$;
in particular, no vertex is mapped to $u_2$.

To finish the proof, we distinguish two cases based on whether $v_2=(-1,0)$ or $v_2=(0,1)$.
We start with analyzing the case $v_2=(-1,0)$.
Let $i\in\{1,2\}$ and let $v$ be a neighbor of $v_i$ different from $(0,0)$ and $v_i+v_i$.
If $f(v)=(1,1)$ or $f(v)=(2,0)$,
then the common neighbor of $(0,0)$ and $v$ different from $v_i$ must be mapped to $u_1$ or $u_2$, which is impossible.
Hence, $f(v)=(1,-1)$.
Since the choice of $i$ and $v$ was arbitrary,
it follows that all the four edges $(1,0)(1,1)$, $(1,0)(1,-1)$, $(-1,0)(-1,1)$ and $(-1,0)(-1,-1)$ are mapped to the edge $b_1$;
in particular, no other edge is mapped to $b_1$ or $b_2$.
This implies that the vertex $(2,0)$ is mapped by $f$ to $(1,1)$.
It follows that the vertex $(2,1)$, which is a common neighbor of $(1,1)$ and $(2,0)$,
must be mapped to the unique common neighbor $u_1=(1,0)$ of the vertices $f((1,1))=(1,-1)$ and $f((2,0))=(1,1)$,
which is impossible. This finishes the analysis of the case $v_2=(-1,0)$.

It remains to analyze the case that $v_2=(0,1)$.
If the vertex $(1,-1)$ was mapped to $(2,0)$ or $(1,1)$,
then the vertex $(0,-1)$, which is a common neighbor of $(1,-1)$ and $(0,0)$,
would have to be mapped to $(1,0)$ or $(0,1)$, which is impossible.
Hence, the vertex $(1,-1)$ is mapped by $f$ to itself and the vertex $(0,-1)$ is also mapped to itself.
Since swapping coordinates is a symmetry mapping $v_1$ and $v_2$ between each other,
a symmetric argument yields that the vertex $(-1,0)$ is mapped to $(0,-1)$.

Next, if the vertex $(2,0)$ was mapped to the vertex $(1,1)$,
then the vertex $(2,-1)$, which is a common neighbor of $(2,0)$ and $(1,-1)$,
would have to be mapped to $(1,0)$, which is impossible.
It follows that the vertex $(2,0)$ must be mapped to $(2,0)$ or $(1,-1)$.
We conclude that the edge $b_1$ is mapped to itself and the edge $b_2$ to either $b_1$ or $b_2$.
A symmetric argument yields that the edge $b_3$ is mapped to $b_1$ and the edge $b_4$ to $b_1$ or $b_2$.
In particular, no other edges of $G$ are mapped to any of the edges $b_1$, $b_2$, $b_3$ and $b_4$.
This implies that the vertex $(1,1)$ is mapped by $f$ to itself.
Consequently, the vertex $(2,0)$ is also mapped to itself (otherwise, the vertex $(2,1)$ would have to be mapped to $(1,0)$).

We now prove the following statement for $r = 1, \ldots, \ell/2-1$ by induction on $r$:
all the vertices $(r,1)$, $(r,-1)$ and $(r+1,0)$ are mapped by $f$ to themselves.
We have already established this statement for $r=1$, so it remains to present the induction step.
Fix $r=2,\ldots,\ell/2-1$ and assume that all the vertices $(r-1,1)$, $(r-1,-1)$ and $(r,0)$ are mapped to themselves.
The vertex $(r,1)$, which is a common neighbor of $(r-1,1)$ and $(r,0)$,
must be mapped to a common neighbor of $(r-1,1)$ and $(r,0)$ at the distance $r+1$ from $(0,0)$.
However, the only such vertex is $(r,1)$.
A symmetric argument yields that the vertex $(r,-1)$ is mapped to itself.
Since the vertex $(r+1,0)$ must be mapped to a neighbor of $(r,0)$ at distance $r+1$ from $(0,0)$,
it can only be mapped to one of the vertices $(r,1)$, $(r+1,0)$ and $(r,-1)$.
By symmetry, it is enough to exclude that it is mapped to $(r,1)$.
If this was the case, then the vertex $(r+1,-1)$, which is a common neighbor of $(r,-1)$ and $(r+1,0)$,
must be mapped to $(r,0)$, which is impossible.
Hence, the vertex $(r+1,0)$ is mapped to itself, concluding the proof of the statement.

We have just shown that the vertex $(\ell/2,0)=(-\ell/2,0)$ is mapped to itself;
earlier, we have shown that the vertex $(-1,0)$ is mapped to $(0,-1)$.
However, the path $(-1,0)(-2,0)\cdots(-\ell/2,0)$ must be mapped by $f$ to a walk with at most $\ell/2$ vertices
but there is no such walk between the vertices $(0,-1)$ and $(-\ell/2,0)$.
Hence, there is no homomorphism from $H(u_0,v_1,v_2)$ to $G(u_0,u_1,u_1)$ with weight at least $\gamma^W$.
\end{proof}

\section{General Condition}
\label{sec-general}

We now present our general technique for establishing that certain graphs do not have the step Sidorenko property.
One difference is that instead of considering only two neighbors of a distinguished vertex $u_0$,
we can choose any number of neighbors $u_1,\dots,u_k$, giving a larger matrix.
More importantly, we are able to restrict homomorphisms considered in the statement
to only those that
map the neighborhood of each $u_i$ bijectively (to the neighborhood of the image of $u_i$, or a chosen subset of it).

The proof extends the arguments presented in the proof of Theorem~\ref{thm-weights}.
The main new idea is that by blowing up $u_i$,
and appropriately choosing weights on copies of the edges to its neighbors,
we can obtain an expression that is counting homomorphisms to the original graph,
but with a weight that is an arbitrary function of how many neighbors of $u_i$ map to each neighbor of the image of $u_i$.
We choose this function to ensure that exactly one neighbor of $u_i$ (or exactly zero) must map to each neighbor of its image.

\begin{theorem}
\label{thm-general}
Let $H$ be a graph and
let $G$ be a weighted graph with $k+1$ distinguished vertices $u_0,u_1,\ldots,u_k$ such that $u_0u_1,\ldots,u_0u_k$ are edges and
$u_1,\ldots,u_k$ form an independent set.
Further, let $U_i$, $i=1,\ldots,k$, be a subset of neighbors of $u_i$ containing $u_0$, and
let $M$ be the $(k\times k)$-matrix such that
the entry $M_{ij}$ is the sum of the weights of homomorphisms from $H(v_0,v_1,v_2)$ to $G(u_0,u_i,u_j)$,
where the sum runs over all choices of vertices $v_0$, $v_1$ and $v_2$ in $H$,
such that the neighbors of $v_1$ are one-to-one mapped to $U_i$ and the neighbors of $v_2$ to $U_j$.
If the matrix $M$ is not positive semidefinite, then $H$ does not have the step Sidorenko property.
\end{theorem}

\begin{proof}
Suppose that the matrix $M$ is not positive semidefinite and fix a vector $a$ such that $a^TMa<0$.
Let $w$ be the weight function of $G$.
As in the proof of Theorem~\ref{thm-weights},
we assume that the sum of the weights of vertices of $G$ is one.
Similarly,
we assume that the weight of each edge is at most $1/2$ (if needed, we can multiply the weights of all edges by the same constant).

We next define a weighted graph $G_{\varepsilon,\alpha}$, which is parameterized by $\varepsilon>0$ and $\alpha\in\RR$.
The structure of the graph is independent of $\varepsilon$ and $\alpha$ and is the following.
Let $n$ be the number of vertices of $H$.
We consider the $3$-blow-up of a vertex $u_0$ and $\left(n^{|U_i|-1}+1\right)$-blow-up of a vertex $u_i$.
The three clones of $u_0$ will be denoted by $u'_0$, $u^+_0$ and $u^-_0$;
one of the $n^{|U_i|-1}+1$ clones of $u_i$ will be denoted by $u'_i$ and
the remaining ones by $u_{i,j_1,\ldots,j_{|U_i|-1}}$ where $1\le j_1,\ldots,j_{|U_i|-1}\le n$.
We next remove every edge going from the vertex $u_{i,j_1,\ldots,j_{|U_i|-1}}$
to a vertex outside the set $U_i$ that is not $u^+_0$ or $u^-_0$,
i.e., the vertex $u_{i,j_1,\ldots,j_{|U_i|-1}}$ is adjacent to $u^+_0$, $u^-_0$ and the vertices of $U_i\setminus\{u_0\}$.

The weight of the vertex $u'_0$ is $(1-2\varepsilon)w(u_0)$ and
the weight of each of the vertices $u^+_0$ and $u^-_0$ is $\varepsilon w(u_0)$.
The weight of the vertex $u'_i$ is $(1-n^{|U_i|-1}\varepsilon)w(u_i)$ and
the weight of each of the vertices $u_{i,j_1,\ldots,j_{|U_i|-1}}$ is $\varepsilon w(u_i)$.
The remaining vertices of $G_{\varepsilon,\alpha}$ have the same weights as in $G$.

Before defining the weights of the edges, we define an auxiliary matrix $B$.
The matrix $B$ has $n$ rows and $n$ columns and $B_{ij}=2^{(i-1)(j-1)}$.
Note that $B$ is a Vandermonde matrix.
Since the matrix $B$ is invertible, there exists a vector $b$ such that $Bb=(0,1,0,\ldots,0)^T$.
The weight of the edge between $u^+_0$ and $u_{i,j_1,\ldots,j_{|U_i|-1}}$
is equal to
\[w(u_0u_i)\left(1+a_i\alpha\prod_{m=1}^{|U_i|-1}b_{j_m}\right)\;\mbox{,}\]
and the weight of the edge between $u^-_0$ and $u_{i,j_1,\ldots,j_{|U_i|-1}}$
is equal to
\[w(u_0u_i)\left(1-a_i\alpha\prod_{m=1}^{|U_i|-1}b_{j_m}\right)\;\mbox{.}\]
The weights of the edges incident with $u'_0$ and
the remaining edges incident with $u^+_0$ and $u^-_0$ are equal to the weights of their counterparts in $G$.
Fix $i\in\{1,\ldots,k\}$ and let $z_1,\ldots,z_{|U_i|-1}$ be the vertices of $U_i$ different from $u_0$.
The weight of the edge between the vertices $u_{i,j_1,\ldots,j_{|U_i|-1}}$ and $z_m$ is equal to $2^{j_m-1}w(u_iz_m)$.
The weights of the edges incident with the vertex $u'_i$ are the same as the weights of their counterparts in $G$.
We have just defined the weights of all edges incident with at least one clone.
The weights of the remaining edges are the same as in~$G$.

We analyze $t(H,W_{\varepsilon,\alpha})$ as a function of $\alpha$ for $\alpha,\varepsilon\in(0,1)$.
In particular, we will show that
\begin{equation}
t(H,G_{\varepsilon,\alpha})=t(H,G_{\varepsilon,0})+c_{\varepsilon}\varepsilon^3\alpha^2+O(\varepsilon^4\alpha^2)\label{eq-general-Taylor}
\end{equation}
for a coefficient $c_{\varepsilon}$, which we will estimate.
Since the coefficient $c_{\varepsilon}$ depends on $\varepsilon$,
it is important to emphasize that the constants hidden in big O notation in \eqref{eq-general-Taylor} are independent of $\varepsilon$ and $\alpha$,
i.e., the equality \eqref{eq-general-Taylor} represents that
there exists $K>0$, which is independent of $\varepsilon$, and a coefficient $c_{\varepsilon}$ for every $\varepsilon\in(0,1)$ such that
the value of $t(H,G_{\varepsilon,\alpha})$ differs from $t(H,G_{\varepsilon,0})+c_{\varepsilon}\varepsilon^3\alpha^2$ by at most $K\varepsilon^4\alpha^2$ for every $\alpha\in(0,1)$.

We now proceed with analyzing the function $t(H,W_{\varepsilon,\alpha})$.
As in the proof of Theorem~\ref{thm-weights},
we observe that $t(H,W_{\varepsilon,\alpha})$ is a polynomial in $\alpha$ and
the linear terms in $\alpha$ cancel out by pairing homomorphisms using $u^+_0$ and those using $u^-_0$.
Hence, only quadratic and higher order terms remain.
To estimate $c_{\varepsilon}$, we need to consider the terms corresponding to homomorphisms mapping exactly three vertices of $H$
to the vertices of $G_{\varepsilon,\alpha}$ with weight $\varepsilon$ and
these vertices must induce a $2$-edge path with the middle vertex mapped to $u^+_0$ or to $u^-_0$ (the contribution of
other homomorphisms cancels out by pairing those using $u^+_0$ and those using $u^-_0$,
similarly as in the proof of Theorem~\ref{thm-weights}).
We arrive at the following identity.

\begin{align*}
c_{\varepsilon}\varepsilon^3 = & \sum_{v_0v_1,v_0v_2\in E(H)}\sum_{i,i'=1}^k\sum_{j\in[n]^{|U_i|-1}}\sum_{j'\in [n]^{|U_{i'}|-1}}
a_ia_{i'}\prod_{m=1}^{|U_i|-1}b_{j_m}\prod_{m=1}^{|U_{i'}|-1}b_{j'_m} \times \\
 & \;\left(\sumh(H(v_0,v_1,v_2),G_{\varepsilon,0}(u^+_0,u_{i,j_1,\ldots,j_{|U_i|-1}},u_{i',j'_1,\ldots,j'_{|U_{i'}|-1}}))\right.+\\
 & \;\;\;\left.\sumh(H(v_0,v_1,v_2),G_{\varepsilon,0}(u^-_0,u_{i,j_1,\ldots,j_{|U_i|-1}},u_{i',j'_1,\ldots,j'_{|U_{i'}|-1}}))\right)
\end{align*}
\pagebreak[3]
It follows that
\begin{align*}
\lim\limits_{\varepsilon\to 0}c_{\varepsilon}=
\sum\limits_{\begin{smallmatrix}v_0v_1\\v_0v_2\end{smallmatrix}\in E(H)}\sum\limits_{i,i'=1}^k\sum\limits_{h}\ 2a_i&a_{i'}w(h)
\sum\limits_{j\in [n]^{|U_i|-1}}\sum\limits_{j'\in[n]^{|U_{i'}|-1}}\\&
\prod\limits_{m=1}^{|U_i|-1}b_{j_m}2^{(j_m-1)h(v_1\hookrightarrow z_m)}\prod\limits_{m=1}^{|U_{i'}|-1}b_{j'_m}2^{(j'_m-1)h(v_2\hookrightarrow z'_m)}
\end{align*}
where the sum is taken over all homomorphisms $h$ from $H$ to $G$ such that $h(v_0)=u_0$, $h(v_1)=u_i$ and $h(v_2)=u_{i'}$, and
$w(h)$ denotes the weight of the homomorphism $h$,
$h(v_1\hookrightarrow z_m)$ denotes the number of neighbors of $v_1$ mapped to $z_m\in U_i$ and
$h(v_2\hookrightarrow z'_m)$ denotes the number of neighbors of $v_2$ mapped to $z'_m\in U_{i'}$.
Observe that $b$ was chosen so that the expression
\[\sum_{j_1,\ldots,j_{|U_i|-1}=1}^n\prod_{m=1}^{|U_i|-1}b_{j_m}2^{(j_m-1)h(v_1\hookrightarrow z_m)}=
  \prod_{m=1}^{|U_i|-1}\sum_{j_m=1}^nb_{j_m}2^{(j_m-1)h(v_1\hookrightarrow z_m)}\]
is one if $h(v_1\hookrightarrow z_m)=1$ and it is zero otherwise.
Hence, it follows that
\[\lim_{\varepsilon\to 0}c_{\varepsilon}=\sum\limits_{v_0v_1,v_0v_2\in E(H)}\sum\limits_{i,i'=1}^k\sum\limits_{h}a_ia_{i'}w(h)\]
where the sum is taken over homomorphisms $h$ from $H$ to $G$ such that $h(v_0)=u_0$, $h(v_1)=u_i$, $h(v_2)=u_{i'}$,
all neighbors of $v_1$ are one-to-one mapped to $U_i$ and all neighbors of $v_2$ are one-to-one mapped to $U_{i'}$.
The definition of the matrix $M$ now implies that
\begin{equation}
\lim_{\varepsilon\to 0}c_{\varepsilon}=\sum\limits_{i,i'=1}^k M_{ii'}a_ia_{i'}=a^T Ma<0\;\mbox{.}
\label{eq-general-lim}
\end{equation}
The expressions \eqref{eq-general-Taylor} and \eqref{eq-general-lim} imply that
there exist $\varepsilon>0$ and $\alpha>0$ such that $t(H,G_{\varepsilon,\alpha})<t(H,G_{\varepsilon,0})$.
Fix such $\varepsilon$ and $\alpha$ for the rest of the proof.

Consider the graphons $W_0$ and $W_{\alpha}$ associated
with the weighted graphs $G_{\varepsilon,0}$ and $G_{\varepsilon,\alpha}$, respectively.
Let $J_u$ be the measurable set corresponding to the vertex $u$ of $G_{\varepsilon,0}$;
we can assume that the measurable set corresponding to the vertex $u$ of $G_{\varepsilon,\alpha}$ is also $J_u$.
Let $\PP$ be the partition of $[0,1]$ formed by $J_{u^+_0}\cup J_{u^-_0}$ and $J_u$, $u\not=u^+_0,u^-_0$.
Observe that $W_0=W^\PP_{\alpha}$.
Since $t(H,W_0)=t(H,G_{\varepsilon,0})$ and $t(H,W_{\alpha})=t(H,G_{\varepsilon,\alpha})$,
we conclude that the graph $H$ does not have the step Sidorenko property.
\end{proof}

Theorem~\ref{thm-general} yields immediately the following corollary,
which in particular rules out many non-biregular graphs to have the step Sidorenko property.
Note that the assumptions of the corollary are easy to verify.

\begin{corollary}
\label{cor-cherry}
Let $H$ be a graph and ${\cal D}_H$ the set of degrees of its vertices.
Further let $M$ be the matrix with rows and columns indexed by the elements of ${\cal D}_H$ such that
the entry $M_{dd'}$ is equal to the number of $2$-edge paths from a vertex of degree $d$ to a vertex of degree $d'$ in $H$.
If the matrix $M$ is not positive semidefinite, then $H$ does not have the step Sidorenko property.
\end{corollary}

\begin{proof}
We can assume without loss of generality that $H$ is bipartite;
if not, $H$ does not even have the Sidorenko property.
Let $n=|H|$, let $d_1<\cdots<d_k$ be the degrees of vertices of $H$, i.e., ${\cal D}_H=\{d_1,\ldots,d_k\}$, and let $D=d_1+\cdots+d_k$.
We next construct a weighted bipartite graph $G_{\varepsilon}$ with weights depending on a parameter $\varepsilon>0$.
One part of $G_{\varepsilon}$ has $k+1$ vertices, which are denoted by $u_1,\ldots,u_{k+1}$, and
the other part has $D-k+1$ vertices. One of the vertices of the second part is denoted by $u_0$ and
the remaining $D-k$ vertices are split into disjoint sets $U_1,\ldots,U_k$ such that $|U_i|=d_i-1$, $i=1,\ldots,k$.
The vertices $u_0$ and $u_{k+1}$ have weight one,
each of the vertices $u_i$ has weight $\varepsilon^{\frac{1}{|U_i|}}$ and
each vertex contained in a set $U_i$ has weight $\varepsilon^{\frac{1}{|U_i|}}/(|U_i|-1)!$, $i=1,\ldots,k$.
The weights of all edges of $G_{\varepsilon}$ are equal to one.

We will apply Theorem~\ref{thm-general} with the weighted graph $G_{\varepsilon}$, vertices $u_0,\ldots,u_k$ and sets $U_1\cup\{u_0\},\ldots,U_k\cup\{u_0\}$.
Let $M_{\varepsilon}$ be the matrix from the statement of Theorem~\ref{thm-general} for the graph $G_{\varepsilon}$.
Fix $i,j\in\{1,\ldots,k\}$ and a $2$-edge path $v_1v_0v_2$ such that the degree of $v_1$ is $d_i$ and the degree of $v_2$ is $d_j$.
Let $h$ be a mapping such that $h(v_0)=u_0$, $h(v_1)=u_i$ and $h(v_2)=u_j$.
The mapping $h$ can be extended to $(|U_i|-1)!(|U_j|-1)!$ homomorphisms from $H$ to $G$ such that
\begin{itemize}
\item the neighbors of $v_1$ are one-to-one mapped to $U_i\cup\{u_0\}$,
\item the neighbors of $v_2$ are one-to-one mapped to $U_j\cup\{u_0\}$, and
\item all other vertices of $H$ are mapped to $u_0$ or to $u_{k+1}$.
\end{itemize}
Each such homomorphism has weight $\frac{\varepsilon^2}{(|U_i|-1)!(|U_j|-1)!}$, i.e., their total weight is $\varepsilon^2$.
Any other extensions of $h$ to a homomorphism from $H$ to $G$ such that
the neighbors of $v_1$ are one-to-one mapped to $U_i\cup\{u_0\}$ and the neighbors of $v_2$ to $U_j\cup\{u_0\}$
has weight at most $\varepsilon^{2+1/d_k}$.
We conclude that the entry of the matrix $M_{\varepsilon}$ in the $i$-th row and the $j$-th column
is equal to $M_{ij}\varepsilon^2+O(\varepsilon^{2+1/d_k})$.
It follows that there exists $\varepsilon>0$ such that the matrix $M_{\varepsilon}$ is not positive semidefinite.
Theorem~\ref{thm-general} now yields that $H$ does not have the step Sidorenko property.
\end{proof}

The weights of vertices and edges of the graph $G$ in Theorem~\ref{thm-general} can be set to lower the weight of
specific homomorphisms, as we did in Corollary~\ref{cor-grid}.
We first formalize the ideas used there, so that we can focus on just the existence of very restricted homomorphisms, without counting or weights.

\begin{lemma}
\label{lem-distances}
Let $H$ be a vertex-transitive graph.
Let $u_0$, $u_1$ and $u_2$ be (distinct) distinguished vertices in $H$ such that $u_0u_1$ and $u_0u_2$ are edges.
Suppose that for each distinct neighbors $v_1$ and $v_2$ of $u_0$,
there is no homomorphism $f$ from $H(u_0,v_1,v_2)$ to $H(u_0,u_1,u_1)$ that simultaneously satisfies the following:
\begin{itemize}
\item neighbors of $v_i$ are one-to-one mapped to neighbors of $u_1$ for $i=1,2$,
\item distances from $u_0$ are preserved, i.e., $\dist(v,u_0) = \dist(f(v),u_0)$ for each $v \in V(H)$, and
\item no vertex other than $u_0$, $v_1$ and $v_2$ is mapped to any of $u_0$, $u_1$ and $u_2$.
\end{itemize}
Then $H$ does not have the step Sidorenko property.
\end{lemma}
\begin{proof}
We start with constructing a weighted graph $G_{\gamma}$ where the weights depend on a parameter $\gamma\in\NN$.
The graph $G_{\gamma}$ is obtained from $H$
by setting $w(v):=\gamma^{\dist(u_0,v)-1}$ for $v \in \{u_0,u_1,u_2\}$ and
$w(v) := \gamma^{\dist(u_0,v)}$ for each vertex $v\not=u_0,u_1,u_2$.
The weights of all edges of $G_{\gamma}$ are one.
We apply Theorem~\ref{thm-general} to $H$ and $G_\gamma$ with the distinguished vertices $u_0$, $u_1$ and $u_2$.
Since $H$ is vertex-transitive,
we will analyze the matrix $M$ such that
$M_{ij}$ is the sum of weights of homomorphisms from $H(u_0,v_1,v_2)$ to $G_\gamma(u_0,u_i,u_j)$ such that
the neighbors of $v_1$ and $v_2$ are mapped one-to-one to the neighbors of $u_i$ and $u_j$, respectively,
where the sum runs over all choices of $v_1$ and $v_2$ in $H$.
Note that the matrix from the statement Theorem~\ref{thm-general} is the considered matrix $M$ with each entry multiplied by $|G|$,
in particular, it is enough to show that the considered matrix $M$ is not positive semidefinite for some $\gamma$.

Let $W :=\sum_{v \in V(H)} \dist(v,u_0) - 3$.
We show that $M_{1,1} = o(\gamma^W)$, $M_{1,2}=M_{2,1}=\Omega(\gamma^W)$ and $M_{2,2} = O(\gamma^W)$ (as functions of the parameter $\gamma$).
Hence, if $\gamma$ is large enough, the matrix $M$ is not positive semidefinite and
$H$ does not have the step Sidorenko property by Theorem~\ref{thm-general}.

By the definition, the entry $M_{1,2}$ contains
a summand corresponding to the identity homomorphism from $H(u_0,v_1,v_2)$ to $G_\gamma(u_0,u_1,u_2)$;
the weight of this summand is exactly $\gamma^W$. It follows $M_{1,2}=M_{2,1}=\Omega(\gamma^W)$.

Consider a homomorphism $f$ contributing to the sum defining the entry $M_{i,i}$ for $i\in\{1,2\}$.
Observe that $f$ satisfies $\left|f^{-1}(\{u_0,u_1,u_2\})\right|\ge 3$ (at least the three vertices $u_0$, $v_1$ and $v_2$
are mapped to $u_0$ and $u_i$) and $\dist(u_0,f(v))\leq\dist(u_0,v)$ for every vertex $v$ (a shortest walk from $u_0$ to $v$
is mapped by $f$ to a walk of at most the same length from $u_0$ to $f(v)$).
Hence, it holds that $w(f(v))\le w(v)$ for every vertex $v$, and
the equality holds for all vertices $v$ if and only if
$\dist(u_0,f(v))=\dist(u_0,v)$ for every vertex $v$ of $H$ and $\left|f^{-1}(\{u_0,u_1,u_2\})\right|=3$.
In particular, the equality does not hold for any homomorphism $f$ contributing to the sum defining the entry $M_{1,1}$.
It follows that each summand in the sum defining the entry $M_{1,1}$ is of order $O(\gamma^{W-1})$ and
each summand in the sum defining the entry $M_{2,2}$ is of order $O(\gamma^W)$.
Since the number of the summands is independent of $\gamma$,
we conclude that $M_{1,1}=o(\gamma^W)$ and $M_{2,2}=O(\gamma^W)$.
\end{proof}

We conclude the paper with applying Lemma~\ref{lem-distances} to show that
all multidimensional grids other than hypercubes are not weakly norming.

\begin{corollary}
\label{cor-multigrid}
Let $k\geq 2$. 
The Cartesian product $C_{\ell_1}\cart\cdots\cart C_{\ell_k}$
has the step Sidorenko property if and only if the length of each cycle in the product is four,
i.e., $\ell_1=\cdots=\ell_k=4$.
\end{corollary}

\begin{proof}
Let $H=C_{\ell_1}\cart\cdots\cart C_{\ell_k}$.
By symmetry, we can assume that $\ell_1$ is the largest and $\ell_2$ is the smallest among $\ell_1,\ldots,\ell_k$.
If $\ell_1=\cdots=\ell_k=4$, the graph $H$ is isomorphic to the $2k$-dimensional hypercube graph,
which is weakly norming, see~\cite{Hat10} and~\cite[Proposition 14.2]{Lov12};
this implies implies that $H$ has the step Sidorenko property~\cite[Proposition 14.13]{Lov12}.
If $\ell_i$ is odd for some $i$, then the graph $H$ is not bipartite,
which implies that it fails to even have the Sidorenko property.
Hence, we can assume that all $\ell_i$ are even and $\ell_1>4$.

We will view the vertices of $H$ as the elements of $\ZZ_{\ell_1}\times\cdots\times\ZZ_{\ell_k}$ and
perform all computations involving the $i$-th coordinate modulo $\ell_i$.
Let $e_i$ be the $i$-th unit vector.
Note that two vertices of $H$ are adjacent if their difference is equal to $e_i$ or $-e_i$ for some $i=1,\ldots,k$.
Also observe that if $v$ is a vertex of $H$ and $\ell_i>4$, then $v$ is the only common neighbor of $v+e_i$ and $v-e_i$.

We apply Lemma~\ref{lem-distances} with $u_0=(0,\ldots,0)$ and $u_i=e_i$ for $i=1,2$.
Suppose that for some distinct vertices $v_1$ and $v_2$,
there is a homomorphism $f$ from $H(u_0,v_1,v_2)$ to $H(u_0,e_1,e_1)$
contradicting the assumption of Lemma~\ref{lem-distances}, i.e.,
\begin{enumerate}[label=(\roman*)]
\item\label{r1} the neighbors of $v_i$ are one-to-one mapped to neighbors of $e_1$, for $i=1,2$,
\item\label{r2} $\dist(u_0,v) = \dist(u_0,f(v))$ for each $v \in V(H)$, and
\item\label{r3} no vertex other than $u_0$, $v_1$ and $v_2$ is mapped to any of the vertices $u_0$, $e_1$ and $e_2$.
\end{enumerate}
We will show that the existence of such a homomorphism $f$ leads to a contradiction.
By symmetry, we can assume that $v_1=e_{i_1}$ for some $i_1$ and
either $v_2=-e_{i_1}$ or $v_2=e_{i_2}$ for some $i_2\not=i_1$.

Note that the neighbors of $v_1$ are one-to-one mapped to the neighbors of $e_1$, and
let $i'$ be such that $f(e_{i_1}+e_{i'})=e_1+e_1$.
If $i'\not=i_1$, both common neighbors of $u_0$ and $e_{i_1}+e_{i'}$, which are $e_{i_1}$ and $e_{i'}$,
must be mapped to the unique common neighbor of $u_0$ and $e_1+e_1$, which is the vertex $e_1$ (note that $\ell_1>4$).
However, this would contradict~\ref{r3}. Hence, $i'=i_1$, i.e., $f(v_1+v_1)=f(e_{i_1}+e_{i_1})=e_1+e_1$.
It follows that there exists a bijection $\pi$ between $\{\pm e_{i'} \mid i' \neq i_1\}$ and $\{\pm e_{j'} \mid j'\neq 1\}$ such that
$f(e_{i_1} + e) = e_1 + \pi(e)$ for $e \in \{\pm e_{i'} \mid i' \neq i_1\}$.
Observe that a symmetric argument to the one that we have just presented yields that $f(v_2+v_2)=e_1+e_1$.

To exclude the case that $v_2=-e_{i_1}$, let $e=\pi^{-1}(e_2)$, i.e., $f(e_{i_1}+e)=e_1+e_2$.
Note that $e\not=\pm e_{i_1}$.
It follows that the vertex $e$, which is a common neighbor of $u_0$ and $e_{i_1}+e$,
must be mapped to a common neighbor of $u_0$ and $e_1+e_2$, i.e., either to $e_1$ or to $e_2$.
The first case would contradict~\ref{r3}, hence $e$ is mapped to $e_2$, meaning $v_2 = e$. We conclude that $v_2=e_{i_2}$ for some $i_2\not=i_1$ and that $f(e_{i_1}+e_{i_2})=e_1+e_2$.
\medskip

Suppose that $\ell_2=4$ and recall that $f(v_2 + v_2) = e_1 + e_1$.
If additionally $\ell_{i_2}=4$, then $-e_{i_2}$, which is a common neighbor of $u_0$ and $e_{i_2}+e_{i_2}$,
must be mapped to the unique common neighbor of $u_0$ and $e_1+e_1$, i.e., to the vertex $e_1$;
this is impossible by~\ref{r3}. Hence, $\ell_{i_2}\not=4$.

Let us call two vertices $v$ and $v'$ \emph{close} if they have at least two common neighbors.
Observe that two close distinct neighbors $v$ and $v'$ of $e_{i_1}$
must be mapped to close neighbors of $e_1$;
otherwise, all common neighbors of $v$ and $v'$ would be mapped to $e_{i_1}$, contradicting~\ref{r3}.
Since the neighborhood of $e_{i_1}$ is one-to-one mapped to the neighborhood of $e_1$ and
the number of pairs of close neighbors of $e_{i_1}$ is the same as the number of pairs of close neighbors of $e_1$,
it follows that pairs of close neighbors of $e_{i_1}$ are one-to-one mapped to pairs of close neighbors of $e_1$ and
pairs of non-close neighbors of $e_{i_1}$ are one-to-one mapped to pairs of non-close neighbors of $e_1$.
Since $\ell_{i_2}\not=4$, the neighbors $e_{i_1}+e_{i_2}$ and $e_{i_1}-e_{i_2}$ of $e_{i_1}$ are not close.
On the other hand, since $\ell_2=4$,
the vertex $f(e_{i_1}+e_{i_2})=e_1+e_2$ has a common neighbor other than $e_1$ with each neighbor of $e_1$.
In particular, $f(e_{i_1}+e_{i_2})$ and $f(e_{i_1}-e_{i_2})$ are close, which is impossible.
We conclude that $\ell_2\not=4$.
Since $\ell_2$ is the smallest among $\ell_1,\ldots,\ell_k$, it follows that each $\ell_i$ is at least six.
\medskip
	
As the final step of the proof of the corollary, we prove the following statement for $r=1,\dots,\ell_{i_1}/2$
by induction on $r$:
\begin{gather}
	f((r-1) e_{i_1}) = (r-1) e_1, \quad f(r e_{i_1}) = r e_1,\text{ and}\nonumber\\
	f(r e_{i_1} + e) = r e_1 + \pi(e)\ \text{ for } e \in \{\pm e_{i'} \mid i'\neq i_1\}.
\label{eqmultigrid}
\end{gather}
The case $r=1$ follows from the definition of $i_1$ and $\pi$.
We assume that the above statement holds for $r$ and prove it for $r+1 \leq \ell_{i_1}/2$.
We first show that $f((r+1) e_{i_1}) = (r+1) e_1 $.
Note that $f(r e_{i_1} + e_{i_1})$ cannot be $r e_1 -e_1$ by \ref{r2}.
If $f(r e_{i_1} + e_{i_1})$ is $r e_1  + e_j$ for some $j\neq 1$,
then the common neighbor $r e_{i_1} + e_{i_1} + \pi^{-1}(-e_{j})$ of $r e_{i_1} + e_{i_1}$ and $r e_{i_1} + \pi^{-1}(-e_{j})$
must be mapped to the unique common neighbor of $r e_1  + e_j$ and $r e_1  - e_j$, which is $r e_1$, contradicting \ref{r2}.
An analogous argument excludes that $f(r e_{i_1} + e_{i_1})$ is $r e_1  - e_j$ for some $j\neq 1$.
Since the vertex $f((r+1) e_{i_1})$ must be a neighbor of $f(re_{i_1})=re_1$, it follows that $f((r+1) e_{i_1}) = (r+1) e_1 $.

We next analyze $f((r+1) e_{i_1} + e)$ for $e\neq\pm e_{i_1}$.
Since the vertex $(r+1) e_{i_1} + e = r e_{i_1} + e_{i_1} + e$
is a common neighbor of $r e_{i_1} + e_{i_1}$ and $r e_{i_1} +  e$,
it must be mapped to a common neighbor of $r e_1 + e_1$ and $r e_1 + \pi(e)$,
i.e., to $r e_1$ or $r e_1 + e_1 + \pi(e)$.
Since the former is excluded by \ref{r2}, it follows that $f((r+1) e_{i_1} +e)=(r+1)e_1 + \pi(e)$.
This concludes the proof of \eqref{eqmultigrid}.

The statement \eqref{eqmultigrid} implies that $f(\ell_{i_1}/2 \cdot e_{i_1}) = \ell_{i_1}/2 \cdot e_1$,
in particular $\ell_{i_1} \geq \ell_1$ by \ref{r2}.
Since the path $u_0, -e_{i_1},-2e_{i_1}, \ldots,-\ell_{i_1}/2 \cdot e_{i_1}$
must be mapped to a path from $u_0$ to $f(-\ell_{i_1}/2 \cdot e_{i_1})=f(\ell_{i_1}/2 \cdot e_{i_1})=\ell_{i_1}/2 \cdot e_1$ and
the vertices of the path must be mapped to vertices at distances $0,1,\ldots,\ell_{i_1}/2$ from $u_0$ by \ref{r2},
the path can be mapped only to the path $u_0,e_1,2e_1,\ldots, \ell_{i_1}/2 \cdot e_{1}$ or,
if $\ell_1=\ell_{i_1}$, to the path $u_0,-e_1,-2e_1,\ldots,-\ell_{i_1}/2 \cdot e_{1}$ 
The former case is impossible since $-e_{i_1}$ cannot be mapped to $e_1$ by \ref{r3}.
It follows that $\ell_1=\ell_{i_1}$ and $f(-e_{i_1})=-e_1$.
Hence, the vertex $e_{i_2}-e_{i_1}\not=u_0$, which is a common neighbor of $e_{i_2}$ and $-e_{i_1}$,
must be mapped to the unique common neighbor of $f(e_{i_2}) = e_1$ and $f(-e_{i_1})=-e_1$,
which is $u_0$. However, this contradicts~\ref{r3}.
We conclude there is no homomorphism $f$ satisfying \ref{r1}--\ref{r3}.
Lemma~\ref{lem-distances} now implies that $H$ does not have the step Sidorenko property.
\end{proof}

\section{Conclusion}

Corollary~\ref{cor-grid} and Corollary~\ref{cor-multigrid} give an infinite class of edge-transitive graphs 
that are not weakly norming, which answers in the negative a question of Hatami~\cite{Hat10}.
Conlon and Lee~\cite[Conjecture 6.3]{ConL17} present a large class of weakly norming graphs,
which they call reflection graphs, and
conjecture that a bipartite graph is weakly norming if and only if it is edge-transitive under a subgroup of its automorphism group (generated by so called `cut involutions').
In particular, this would imply that all weakly norming graphs are edge-transitive.

Since every weakly norming graph has the step Sidorenko property,
it is natural to ask whether the converse is true for connected graphs, i.e.,
whether every connected graph with the step Sidorenko property is weakly norming.
This question has been very recently answered in the affirmative by Dole\v zal et al.~\cite{DolGHRR18} who showed the following:
a connected graph $G$ is weakly norming if and only if it has the step Sidorenko property.

Finally, it is natural to wonder about the Forcing Conjecture in the setting of the step Sidorenko property.
Let us say that a graph $H$ has the {\em step forcing property} if and only if
\[t(H,W^{\PP})\le t(H,W)\]
for every graphon $W$ and every partition $\PP$ of $[0,1]$ into finitely many non-null measurable sets and
the equality holds if and only if $W^{\PP}$ and $W$ are equal almost everywhere.
It can be shown that all even cycles have the step forcing property (while an ad hoc argument can be given,
this also follows from~\cite[Theorem 3.14]{DolGHRR18}).
Graphs with the step forcing property
are related to the proof of the existence of graphons via weak$^{*}$ limits given by Dole\v zal and Hladk\'y~\cite{DolH17};
in particular, if $H$ has the step forcing property,
minimizing the entropy of $W$ in the arguments given in~\cite{DolH17} can be replaced by maximizing $t(H,W)$.

\section*{Acknowledgements}

The first author would like to thank Karel Kr\'al and L\'aszl\'o Mikl\'os Lov\'asz
for discussions on the step Sidorenko property and the step forcing property.

\bibliographystyle{bibstyle}
\bibliography{stepsidorenko}

\end{document}